\documentclass[a4paper,12pt]{article}

\newenvironment{proof}{\noindent {\bf Proof:}}{\hfill $\Box$}

\newtheorem{theorem}{Theorem}
\newtheorem{lemma}{Lemma}

\newtheorem{assumption}{Assumption}

\usepackage{amsmath}
\usepackage{amssymb}
\usepackage{amsfonts}
\usepackage{graphicx}
\usepackage{mathrsfs}

\title{\bf Strong duality in Lasserre's hierarchy for polynomial optimization}

\begin{document}

\author{C\'edric Josz$^{1,2}$, Didier Henrion$^{3,4,5}$}

\footnotetext[1]{French transmission system operator R\'eseau de Transport d'Electricit\'e (RTE), 9 rue de la Porte de
Buc, BP 561, F-78000 Versailles, France.}
\footnotetext[2]{INRIA Paris-Rocquencourt, BP 105, F-78153 Le Chesnay, France. {\tt cedric.josz@inria.fr}}
\footnotetext[3]{CNRS, LAAS, 7 avenue du colonel Roche, F-31400 Toulouse, France. {\tt henrion@laas.fr}}
\footnotetext[4]{Universit\'e de Toulouse, LAAS, F-31400 Toulouse, France.}
\footnotetext[5]{Faculty of Electrical Engineering, Czech Technical University in Prague,
Technick\'a 2, CZ-16626 Prague, Czech Republic}

\date{Draft of \today}

\maketitle

\begin{abstract}
A polynomial optimization problem (POP) consists of minimizing a multivariate real polynomial on a semi-algebraic
set $K$ described by polynomial inequalities and equations. In its full generality it is a non-convex,
multi-extremal, difficult global optimization problem.  More than an decade ago, J.~B.~Lasserre
proposed to solve POPs by a hierarchy of convex semidefinite programming (SDP) relaxations
of increasing size. Each problem in the hierarchy has a primal SDP formulation (a relaxation of
a moment problem) and a dual SDP formulation (a sum-of-squares representation
of a polynomial Lagrangian of the POP). In this note, when the POP feasibility set $K$ is compact,
we show that there is no duality gap between each primal and dual SDP problem
in Lasserre's hierarchy, provided a redundant ball constraint is added to the
description of set $K$. Our proof uses elementary results on SDP duality,
and it does not assume that $K$ has an interior point.
\end{abstract}

\section{Introduction}
Consider the following polynomial optimization problem (POP)
\begin{equation}
\label{eq:pop}
\begin{array}{ll}
\inf_x & f(x) := \sum_\alpha f_{\alpha} x^\alpha \\
\mathrm{s.t.} & g_i(x) := \sum_\alpha g_{i,\alpha} x^\alpha \geq 0, \quad i=1,\ldots,m
\end{array}
\end{equation}
where we use the multi-index notation $x^\alpha := x_1^{\alpha_1} \cdots x_n^{\alpha_n}$ for $x \in {\mathbb R}^n$,
$\alpha \in {\mathbb N}^n$ and
where the data are polynomials $f, g_1, \ldots, g_m \in {\mathbb R}[x]$
so that in the above sums only a finite number of coefficients
$f_{\alpha}$ and $g_{i,\alpha}$ are nonzero. Assume that the feasibility set
\[
K:=\{x \in {\mathbb R}^n \: :\: g_i(x) \geq 0, \: i=1,\ldots,m\}
\]
is nonempty and bounded, so that the above infimum is attained. To solve POP (\ref{eq:pop}),
Lasserre \cite{lasserre-2000,lasserre-2001} proposed a semidefinite programming (SDP) relaxation hierarchy
with guarantees of global and generically finite convergence  \cite{nie14} provided an algebraic assumption holds:

\begin{assumption}\label{archimedean}
There exists a polynomial $u \in \mathbb{R}[x]$ such that $\{x\in \mathbb{R}^n  \: :\: u(x)\geq 0\}$ is bounded
and $u=u_0 + \sum_{i=1}^m u_i g_i $ where polynomials $u_i \in {\mathbb R}[x]$, $i=0,1,\ldots,m$
are sums of squares (SOS) of other polynomials.
\end{assumption}

Assumption \ref{archimedean} can be difficult to check computationally (as the degrees of the SOS
multipliers can be arbitrarily large), and it is often replaced
by the following slightly stronger assumption:

\begin{assumption}\label{ball}
Let $R>0$ be the radius of an Euclidean ball including set $K$, and add the
redundant ball constraint $g_{m+1}(x) = R^2 - \sum_{i=1}^n x_i^2 \geq 0$ to the description of $K$.
\end{assumption}

Indeed, under Assumption \ref{ball}, simply choose $u=g_{m+1}$, $u_1=\cdots=u_m=0$, and $u_{m+1}=1$
to conclude that Assumption \ref{archimedean} holds as well.
In practice, it is often easy to identify a bound $R$ on the radius of the feasibility set $K$.

Each problem in Lasserre's hierarchy consists of a primal-dual SDP pair, called SDP relaxation,
where the primal  corresponds
to a convex moment relaxation of the original (typically nonconvex) POP, and the dual  corresponds
to a SOS representation of a polynomial Lagrangian of the POP. The question arises of the absence of
duality gap in each SDP relaxation. This is of practical importance because numerical algorithms
to solve SDP problems are guaranteed to converge only where there is no duality gap,
and sometimes under the stronger assumption that there is a primal or/and dual SDP interior point.

 In \cite[Example 4.9]{schweighofer},
Schweighofer provides a two-dimensional POP with bounded $K$ with no interior point
for which Assumption \ref{archimedean} holds, yet a duality gap exists at the first
SDP relaxation: $\inf x_1x_2 \:\mathrm{s.t.}\: x \in K=\{x \in {\mathbb R}^2 \: :\:
-1 \leq x_1 \leq 1, \: x^2_2 \leq 0\}$, with
primal SDP value equal to zero and dual SDP value equal to minus infinity.
This shows that a stronger assumption is required to ensure no SDP duality gap. 
A sufficient condition for strong duality has been given in \cite{lasserre-2001}:
the interior of the POP feasibility set $K$ should be nonempty. However, this may be too restrictive:
in the proof of Lemma 1 in \cite{hl12} the authors use notationally awkward arguments
involving truncated moment matrices
to prove the absence of SDP duality gap for a set $K$ with no interior point. This shows that absence of an interior point
for $K$ is not necessary for no SDP duality gap, and a weaker assumption is welcome.

Motivated by these observations, in this note we prove that under the basic Assumption \ref{ball}
on the description of set $K$, there is no duality gap in the SDP hierarchy. In particular, this
covers the cases when $K$ has an empty interior. Our interpretation of this result,
and the main message of this contribution, is that in the context of Lasserre's hierarchy
for POP, a practically relevant description of a bounded semialgebraic feasibility set
must include a redundant ball constraint.

\section{Proof of strong duality}

For notational convenience, let $g_0(x)=1 \in {\mathbb R}[x]$ denote the unit polynomial.
Define the localizing moment matrix
\[
M_{d-d_i}(g_iy) := \left(\sum_\gamma  g_{i,\gamma} y_{\alpha+\beta+\gamma}\right)_{|\alpha|,|\beta|\leq d-d_i}
 = \sum_{|\alpha|\leq 2d} A_{i,\alpha} y_{\alpha}
\]
where $d_i$ is the smallest integer greater than or equal to half the degree of $g_i$,
for $i=0,1,\ldots,m$.
For $d \geq d_{\min}:=\max_{i=0,1,\ldots,m} d_i$, the Lasserre hierarchy for POP (\ref{eq:pop})
consists of a  primal moment SDP problem
\[
(P_d) ~:~ 
\begin{array}{ll}
\inf_y & \sum_{\alpha}  f_{\alpha} y_\alpha \\
\mathrm{s.t.} & y_0 = 1 \\
& M_{d-d_i}(g_i y) \succeq 0, \:i=0,1,\ldots,m
\end{array}
\]
and a dual SOS SDP problem
\[
(D_d) ~:~ 
\begin{array}{ll}
\sup_{z,Z} & z \\
\mathrm{s.t.} & f_0 - z =  \sum_{i=0}^m \langle A_{i,0}, Z_i \rangle \\
& f_\alpha = \sum_{i=0}^m \langle A_{i,\alpha}, Z_i \rangle, \quad 0<|\alpha|\leq 2d \\
& Z_i \succeq 0, \:i=0,1,\ldots,m
\end{array}
\]
where $A\succeq 0$ stands for matrix $A$ positive semidefinite, $\langle A,B \rangle = \mathrm{trace}\:AB$ is
the inner product between two matrices.
The primal-dual pair $(P_d,D_d)$ is called the SDP relaxation of order $d$ for POP (\ref{eq:pop}).

Let us define the following sets:
\begin{itemize}
\item $\mathcal{P}_d$: feasible points for $P_d$;
\item $\mathcal{D}_d$: feasible points for $D_d$;
\item $\mathrm{int}\:\mathcal{P}_d$: strictly feasible points for  $P_d$;
\item $\mathrm{int}\:\mathcal{D}_d$: strictly feasible points for $D_d$;
\item $\mathcal{P}^*_d$:  optimal solutions for  $P_d$;
\item $\mathcal{D}^*_d$:  optimal solutions for $D_d$.
\end{itemize}
Finally, let us denote by $\mathrm{val}\:P_d$ the infimum in problem $P_d$
and by $\mathrm{val}\:D_d$ the supremum in problem $D_d$.
Strong duality holds whenever $\mathrm{val}\:P_d=\mathrm{val}\:D_d < \infty$.

\begin{lemma}
\label{lemma:Slater}
$\mathrm{int}\:\mathcal{P}_d$ nonempty or $\mathrm{int}\:\mathcal{D}_d$ nonempty
implies $\mathrm{val}\:P_d=\mathrm{val}\:D_d$.
\end{lemma}

Lemma \ref{lemma:Slater} is classical in convex optimization, and it is generally
called Slater's condition, see e.g. \cite[Theorem 4.1.3]{shapiro-2000}.

\begin{lemma}
\label{lemma:Trnovska}
$\mathcal{P}_d$ is nonempty and $\mathrm{int}\:\mathcal{D}_d$ is nonempty
if and only if $\mathcal{P}^*_d$ is nonempty and bounded.
\end{lemma}

A proof of Lemma \ref{lemma:Trnovska} can be found in \cite{trnovska-2005}.
According to Lemmas \ref{lemma:Slater} and \ref{lemma:Trnovska},
$\mathcal{P}_d^*$ nonempty and bounded implies strong duality.
This result is also mentioned without proof at the end of \cite[Section 4.1.2]{shapiro-2000}.

\begin{lemma}
\label{lemma:bound}
Under Assumption \ref{ball}, set $\mathcal{P}_d$ is included in the Euclidean ball of radius $\sum_{k=0}^{d} R^{2k}$.
\end{lemma}

\begin{proof}
Consider a feasible point $(y_\alpha)_{|\alpha| \leqslant 2d} \in \mathcal{P}_d$.
 Let $k \in {\mathbb N}$ be such that
$1 \leq k \leq d$. In the SDP problem $P_k$,
the localizing matrix associated to the redundant ball constraint $g_{m+1}(x) = R^2 - \sum_{i=1}^n x^2_i \geq 0$ reads
\[
M_{k-1}(g_{m+1} y) = \left(\sum_\gamma g_{m+1,\gamma} ~ y_{\alpha+\beta+\gamma}\right)_{|\alpha|,|\beta|\leq k-1}
\]
with trace equal to
$$ 
\begin{array}{rcl}
\mathrm{trace}\:M_{k-1}(g_{m+1} y) & = & \sum_{|\alpha|\leqslant k-1} \sum_\gamma g_{m+1,\gamma} ~ y_{2\alpha+\gamma} \\\\
  & = & \sum_{|\alpha|\leq k-1} \left( g_{m+1,0} ~ y_{2\alpha} + \sum_{|\gamma|=1} g_{m+1,2\gamma} ~ y_{2\alpha+2\gamma} \right) \\\\
  & = & \sum_{|\alpha|\leq k-1} \left( R^2 y_{2\alpha} - \sum_{|\gamma|=1} y_{2(\alpha+\gamma)} \right) \\\\
  & = & \sum_{|\alpha|\leq k-1} R^2 y_{2\alpha} - \sum_{|\alpha|\leq k-1,|\gamma|=1} y_{2(\alpha+\gamma)} \\\\
  & = & R^2 (\sum_{|\alpha|\leq k-1} y_{2\alpha}) + y_0 - \sum_{|\alpha|\leq k} y_{2\alpha} \\\\
  & = & R^2 \:\mathrm{trace}\:M_{k-1}(y) + 1 -  \mathrm{trace}\:M_{k}(y).\\
\end{array}
$$
From the structure of the localizing matrix, it holds $M_{k-1}(g_{m+1} y) \succeq 0$ hence
$\mathrm{trace}\:M_{k-1}(g_{m+1} y) \geq 0$ and
\[
\mathrm{trace}\:M_k(y) \leq 1 + R^2\:\mathrm{trace}\:M_{k-1}(y)
\]
from which we derive
\[
\mathrm{trace}\:M_d(y) \leq  \sum_{k=1}^{d} R^{2(k-1)} + R^{2d}\:\mathrm{trace}\:M_0(y) =  \sum_{k=0}^{d} R^{2k} 
\]
since $\mathrm{trace}\:M_0(y) = y_0 = 1$.
The norm  $\|M_{d}(y)\|_2$, equal to the maximum eigenvalue of $M_d(y)$, is upper bounded
by $\mathrm{trace}\:M_d(y)$, the sum of the eigenvalues of $M_d(y)$, which are all nonnegative.
Moreover
$$\begin{array}{rcl}
\|M_d(y)\|_2^2 & = & \langle ~ \sum_{|\alpha|\leq 2d} A_\alpha y_\alpha ~,~ \sum_{|\alpha|\leq 2d} A_\alpha y_\alpha ~ \rangle \\\\
 & = & \sum_{|\alpha|\leq  2d} ~ \langle A_\alpha,A_\alpha \rangle ~ y_\alpha^2 ~~~ \text{by orthogonality of the matrices
$(A_\alpha)_{|\alpha|\leq  2d}$}\\\\
 & \geq  & \sum_{|\alpha|\leq  2d} ~ y_\alpha^2 ~~~ \text{because $\langle A_\alpha,A_\alpha \rangle \geq  1$}.
\end{array} 
$$
The proof follows then from
\[
\sqrt{\sum_{|\alpha|\leq  2d} y^2_\alpha} \leq  \|M_d(y)\|_2 \leq \sum_{k=0}^{d} R^{2k}.
\]
\end{proof}

\begin{theorem}
\label{theorem:strong duality}
Under Assumption \ref{ball}, strong duality holds for SDP relaxations of all orders.
\end{theorem}

\begin{proof}
Let us first show that $\mathcal{P}_d$ is nonempty. 
Let $v_d(x):=(x^{\alpha})_{|\alpha|\leq d}$ and
consider a feasible point $x^* \in K$ for POP (\ref{eq:pop}).
Then $y^*=v_{2d}(x^*) \in {\mathcal P}_d$ since by construction
$M_{d-d_i}(g_i y^*) = g_i(x^*) v_{d-d_i}(x^*) v_{d-d_i}(x^*)^T \succeq 0$ for all $i$.
From Lemma \ref{lemma:bound}, $\mathcal{P}_d$ is bounded and closed, and the objective function in $P_d$
is linear, so we conclude that
${\mathcal P}^*_d$ is nonempty and bounded. According to Lemma \ref{lemma:Trnovska}, $\mathrm{int}\:\mathcal{D}_d$
is nonempty, and from Lemma \ref{lemma:Slater} strong duality holds.
\end{proof}

\section{Conclusion}

We prove that strong duality always holds in Lasserre's SDP hierarchy for POP
on bounded semi-algebraic sets after adding a redundant ball constraint.
To preclude numerical troubles with SDP solvers, we advise
to systematically add such a ball constraint, combined with an appropriate scaling
so that all scaled variables belong to the unit sphere. Without scaling,
numerical troubles can occur as well, but they are not due to the presence
of a duality gap, see \cite{hl05} and also the example of \cite[Section 6]{wnm12}.

\end{document}